\documentclass[12pt, reqno]{amsart}
\usepackage[bookmarksnumbered,colorlinks, plainpages]{hyperref}
\hypersetup{colorlinks=true,linkcolor=red, anchorcolor=green, citecolor=cyan, urlcolor=red, filecolor=magenta, pdftoolbar=true}

\usepackage{subfig}
\usepackage{rotating}
\usepackage{tikz}

\usepackage{fixltx2e}
\usepackage{array}
\RequirePackage{tkz-berge}

\usepackage{amsmath, amsthm, amscd, amsfonts, amssymb, graphicx, color, enumerate, xcolor,  mathrsfs, latexsym}

 \definecolor{cupgreen}{rgb}{0,0.498,0.208}
  \definecolor{cupblue}{rgb}{0,0,.5}
  \definecolor{cupred}{rgb}{1,0.04,0}
  \definecolor{cuppink}{rgb}{0.925,0,0.545}
  \definecolor{cupmagenta}{rgb}{0.624,0.161,0.424}
  \definecolor{cupbrown}{rgb}{0.71,0.212,0.133}

  \definecolor{cupgreen}{rgb}{0,0,0}
  \definecolor{cupblue}{rgb}{0,0,0}
  \definecolor{cupred}{rgb}{0,0,0}
  \definecolor{cuppink}{rgb}{0,0,0}
  \definecolor{cupmagenta}{rgb}{0,0,0}
  \definecolor{cupbrown}{rgb}{0,0,0}

\definecolor{TITLE}{rgb}{0,0,0}
\definecolor{AUTHOR1}{rgb}{0.00,0.59,0.00}
\definecolor{AUTHOR2}{rgb}{0.50,0.00,1.00}
\definecolor{SECTION}{rgb}{0.50,0.00,1.00}
\definecolor{FOOTTITLE}{rgb}{0.00,0.50,0.75}
\definecolor{THM}{rgb}{0.8,0,0.1}
\definecolor{SEC}{rgb}{0,0,1}

\makeatletter

\newtheorem{theorem}{{\color{THM} Theorem}}[section]
\newtheorem{lemma}[theorem]{{\color{THM}Lemma}}

\theoremstyle{definition}

\newtheorem{conjecture}[theorem]{{\color{THM}Conjecture\ }}

\numberwithin{equation}{section}

\usepackage{amscd}
\begin{document}

\newbox\Adr
\setbox\Adr\vbox{
\centerline{\sc hamed ghasemian zoeram$^1$ and daniel yaqubi$^2$  }
\vskip18pt
\centerline{$^{1}$Department of Pure Mathematics, Ferdowsi University of Mashhad}
\centerline{P. O. Box 1159, Mashhad 91775, Iran.}
\centerline{Email: {\tt hamed90ghasemian@gmail.com } } 
\vskip18pt
\centerline{$^{2}$Department of Pure Mathematics, Ferdowsi University of Mashhad}
\centerline{P. O. Box 1159, Mashhad 91775, Iran.}
\centerline{Email: {\tt daniel\_yaqubi@yahoo.es }.}} 

\title{Spanning $k$-ended trees of $3$-regular connected graphs}
\author[ Hamed Ghasemian Zoeram, Daniel Yaqubi ]{\box\Adr}

\maketitle
\begin{abstract}
 A vertex of degree one is called an \textit{end-vertex} and  the set of end-vertices of $G$ is denoted by $End(G)$. For
a positive integer $k$, a tree $T$ be called  $k$-ended tree if $\mid End(T )\mid\leq k$. In this paper, for 
each $3$-regular connected graph with $\vert G\vert \geqslant 8$, we find a positive integer 
$k$ related to order of $G$, such that $G$ has a spanning $k$-ended tree, and 
with construction a sequence of $3$-regular connected graphs 
which their orders approach to infinity with known minimum possible number
of end-vertices of their spanning trees. At the end, we  given a 
conjecture about spanning $k$-ended trees of $3$-regular connected graphs.
\end{abstract}

\section{Introduction}
Throughout this article we consider only finite undirected laballed graphs without loops or multiple edges. The vertex set and edge set of graph  $G$ is denoted by $V=V (G)$ and $E=E(G)$, respectively. For $u, v \in V$ An edge joining two vertices $u$ and $v$ is denoted by $uv$ or $vu$. The neighbourhood $N_G(v)$ or $N(v)$ of vertex $v$  is the set of all $u \in V$ which are
adjacent to $v$. The degree of a vertex $v$, denoted by $\deg_G(u)=|N_G(v)|$.
The minimum degree of a graph  $G$  is denoted  $\delta(G)$ and the maximum degree is denoted $\Delta (G)$. 
 If all vertices of $G$ have same degree $k$, then the graph $G$ is called $k$-\textit{regular}.\\ The distance of $u$ and $v$,
denoted by $d_G(u, v) $ or $d(u, v)$, is the length of a shortest path between $u$ and $v$.  A \textit{Hamiltonian path} 
of a graph is a path passing through all vertices of the graph. A graph is \textit{Hamiltonian-connected} if every two 
vertices are connected with a Hamiltonian path. In graph $G$, an \textit{independent set} is a subset $S$ of $V(G)$ such that no two 
vertices in $S$ are adjacent. A maximum independent set is an independent set of largest possible size for a given graph $G$, 
This size is called the independence number of G, that denoted by $\alpha(G)$.  \\
 A vertex of degree one is called an \textit{end-vertex}, and the set of end-vertices of $G$ is denoted by $End(G)$. If
 $T$ be tree, an end-vertex of a $T$ is usually called a \textit{leaf} of $T$ and the set of leaves of $T$ is denoted by $leaf(T)$. A
spanning tree is called independence if $End(G)$ is independent in $G$. For
a positive integer $k$, a tree $T$ is said to be a $k$-ended tree if $\mid End(T )\mid\leq k.$ . We define $\sigma_{k}(G) = \min \{d(v_1)+\ldots+d(v_k) \mid \{v_1, \ldots , v_k \}
\textit{ is an independent set in G}\}$. Clearly, $\sigma_{1}(G) = \delta(G)$.


By using $\sigma_2(G)$, \textit{Ore} \cite{Ore} obtain the following famous theorem on Hamiltonian path. Notice that a Hamiltonian path is spanning $2$-ended tree.  
\begin{theorem} \cite{Ore}
Let $G$ be a connected graph, if $\sigma_2(G)\geq |G|-1$, then $G$ has Hamiltonian path.
\end{theorem}
 The following theorem of  \textit{Las Vergnas Broersma and
Tuinstra} \cite{Bon} gives a similar sufficient condition for a graph $G$ to have a spanning $k$-ended tree. 
\begin{theorem}\cite{Bro}
Let $k\geq2$ be an integer, and let $G$ be a connected graph. If $\sigma_2(G)\geq |G|-k+1$, then $G$ has a spanning $k$-ended tree.
\end{theorem}
\textit{Win } \cite{Win} obtained a sufficient condition related to
independent number for $k$-connected graph that confirms a conjecture of \textit{Las Vergnas Broersma and
Tuinstra} \cite{Bon} gave a degree sum condition for a spanning $k$-ended tree.  \\
\begin{theorem}\cite{Win}
Let $k\geq 2$ and let $G$ be a $m$-connected graph. If $\alpha(G) \leq  m + k -1$, then $G$
has a spanning $k$-ended tree. 
\end{theorem}
A closure operation is useful in the study of existence of Hamiltonian cycles, Hamiltonian path and other spanning subgraphs in graph. 
It was first introduced by \textit{Bondy and Chavatal}.
\begin{theorem}\cite{ Bon}
Let $G$ be a graph and let $u$ and $v$ be two nonadjacent vertices of $G$ then,\\
$(1)$. Suppose $\deg_G(u) +\deg_G(v)\geq |G|$. Then $G$ has a Hamiltonian cycle if and only if $G +uv$ has a Hamiltonian cycle.\\
$(2)$.  Suppose $\deg_G(u) +\deg_G(v)\geq |G|-1$. Then $G$ has a Hamiltonian path if and only if $G +uv$ has a Hamiltonian path.
\end{theorem}
After \cite{Bon}, many researchers have defined other closure concepts for various graph properties.
More on $k$-ended tree and spanning tree can be found in \cite{Aki, Czy, Oze, Sal}.

\section{Our Results}
\begin{lemma}\label{D}
Let $T$ be a tree with $n$ vertices such that $\Delta(T)\leq 3$. If $\mid leaf(T)\mid=k$  and  $p$ be the number of vertices of degree $3$ in $T$,  then $k=p+2$.
\end{lemma}
\begin{proof}
It is easy by the induction on $p$.
\end{proof}
\begin{lemma}\label{DD}
Let $G$ be a labelled graph and $k\geq3$ is the smallest integer such that G has
a spanning tree with k leaves  like $T$, then no two leaves of $T$ are adjacent in $G$.
\end{lemma}
\begin{proof}
Let $T$ be a spanning subtree of $G$ with $k$ leaves. Put $S=\{v_1, v_2, \ldots ,v_k\}$ 
be the set of all leaves of $T$. By contradiction, suppose that $v_1$ and $v_2$  are two leaves in $T$ that adjacent 
in $G$. Consider the garph $T_1= T+ v_1v_2$, then $T_1$ contains a unique cycle as $C:v_1v_2c_1c_2 \ldots c_\ell v_1$  
where $c_i\in G$ for $1\leq i\leq \ell$. Since $k\geq3$ then there exist vertex  $v_s\in G$ such that $v_s\notin C$. 
Let $P$ be the shortest path of vertex $v_s$ to the cycle $C$ such that cross of the cycle $C$ in the vertex $c_j$ 
for $1\leq j\leq\ell$.\\
Now, we omit the edge $c_{j-1}c_j$ of $T_1$, (If $j=1$ put $c_{j-1}=v_2$). Let
 $T_2=T_1-c_{j-1}c_j$. Then $T_2$ is a spanning subtree of $G$ such that
$\deg_{T_2}(c_j)=2$.  The vertices of degree one in spanning subtree $T_2$ is equal to the set $\{v_3, v_4, \ldots ,v_k \}$ either $\{v_3, v_4, \ldots ,v_k, c_{j-1}\}$ . That is a contradict by minimality of $k$.
\end{proof}
\begin{theorem}
Let $G$ be a $3$-regular  connected graph such that $n=|G|\geq 8$. Then $G$ is a $\big\lfloor \frac{2n+4}{9}\big\rfloor$-ended tree.
\end{theorem}
\begin{proof}
For positive integer $k$, let $T$ be a spanning subtree of $G$ with the $k$ leaves such that $k$ is minimum.
If $k=2$ then it is obvious theorem is true, so we suppose $k>2$.\\
Put
\begin{eqnarray*}
 \mathcal{A}_1&=&\{v\in V(G)\quad\vert\quad \deg_T(v)=1\quad \};\\
  \mathcal{A}_2&=&\{v\in V(G)\quad\vert\quad \deg_T(v)=2\quad \};\\
  \mathcal{A}_3&=&\{v\in V(G)\quad\vert\quad \deg(v)=3\quad\}.
\end{eqnarray*}

Let $v$ be a vertex of the set $\mathcal{A}_{1}$. Since $G$ is a $3$-regular graph, 
then there exist two edges as $e_1(=vv_i),e_2(=vv_j)\in E(G)-E(T)$,
adjacent to $v$. By using of lemma ~\ref{DD}, the vertices $v_i$ and $v_j$ belong to the set $\mathcal{A}_2$. 
So, for each
$v\in \mathcal{A}_{1}$ there exist two vertices of the set $\mathcal{A}_{2}$, such that they were adjacent to $v$ in $G$ but not in $T$.\\
Now, choice $v \in \mathcal{A}_{1}$ and consider one of adjacent edges to $v$ like $vw$ that
$w\in \mathcal{A}_2$ and $vw\in E(G)-E(T)$ . The graph
$T+vw$ has a cycle like $vwv_{1} \ldots v_{s}v$ such that $v_{1} \in \mathcal{A}_{2}$ otherwise $v_{1} \in \mathcal{A}_{3}$
and  $T+vw-wv_{1}$ is a spanning subtree of $G$ with $k-1$ leaves, it is impossible.\\
In other hand, no members of $\mathcal{A}_{1}$ except $v$ ( for example $u$) can not be adjacent to $v_{1}$
in $G$ because, then $T_1-wv_{1}$ is a spanning subtree of $G$ with $k$ leaves such in that $v_{1}$
and $u$ have degree one and they are adjacent in $G$, it is contradiction with the lemma ~\ref{DD}.
We can choice the vertex $w$ at above such that the edge $vv_1\notin E(T)$, because if $vv_1\in E(T)$, we 
choice another vertex of the set $\mathcal{A}_2$ like $x$ where $vx\in E(G)-E(T)$. Now, the graph
$T+vx$ contains a cycle like $vxz_1\ldots z_dv$ such that $z_1\neq v_1$ (why). It is  obviously the edge $vz_1\notin E(T)$. So, we can suppose $vv_1 \notin E(T)$.\\
We have two cases:
\begin{itemize}
\item[] \textbf{Case I.} $vv_{1}\notin E(G)$, then we correspond the set $\{v_{1},w, y\}$,($y$ is another vertex 
such that $y\neq w$ and $vy \in E(G)-E(T)$) to $v$.

 \item[] \textbf{Case II.}  If $vv_{1}$ is an edge of $G$ that is not in $T$ then $T+vv_{1}$ has a cycle  like $vv_{1}mu_{1} \ldots u_{\ell}v$ and again like before $m \in A_{2}$. It is true that $m\neq w$ because if not then two cycles $vv_{1}mu_{1} \ldots u_{\ell}v$ an $vwv_{1} \ldots v_{s}v$ must cut each other, suppose that first vertex after $m$ in cycle $vv_{1}mu_{1} \ldots u_{\ell}v$ that is also in $vwv_{1} \ldots v_{s}v$ is $v_{i}=u_{j}$ then $wv_{1} \ldots v_{i}u_{j-1} \ldots u_{1}w$ is a cycle in $T$ and this is not possible.\\
 Now as before no members of $A_{1}$, except $v$ can not be adjacent to $m$ in $G$, then we correspond $\{w, v_{1},m \}$ to $v$.\\
 \end{itemize}
 For other vertices in the set $\mathcal{A}_{1}$, we do corresponding
 like as the vertex $v$. Finally, we correspond for each vertices in the set $\mathcal{A}_{1}$, one set with vertices
 of the set $\mathcal{A}_{2}$, such that each set
 has two elements. For example, for $v\in\mathcal{A}_1$,
 if consider $\mathbf{case\hspace{0.1cm} I}$, then $w$ and $y$ no ones don't appear in no correspondent
  set of no elements of the set $\mathcal{A}_1$ except vertex $v$; 
  and the third element maximum appears in one another correspondent set of another element of the set
  $\mathcal{A}_1$.\\
  So, we have $3\times \vert\mathcal{A}_{1}\vert-\frac{1}{2}\times\vert\mathcal{A}_1\vert \leqslant \vert\mathcal{A}_{2}\vert$. 
  If put $\vert\mathcal{A}_{1}\vert=k ,\vert\mathcal{A}_{2}\vert=n_1$
  and 
 $\vert\mathcal{A}_{3}\vert=p$, by using of the lemma ~\ref{D}, we have $k=p+2$ and since 
 $3\times \vert\mathcal{A}_{1}\vert -\frac{\vert\mathcal{A}_1\vert}{2}\leqslant \vert\mathcal{A}_{2}\vert$, this makes $\frac{5k}{2} \leqslant n_1$.
We have \\
 \[n=p+n_1+k=k-2+n_1+k \geqslant k-2+\frac{5k}{2}+k=\frac{9k}{2}-2 \Longrightarrow k\leqslant \big\lfloor \frac{2n+4}{9}\big\rfloor .\]
 
\end{proof}
\section{Some concluding remarks}

In this section, we construct the sequence $G_m$ of $3$-regular graphs. 
For $m=1$, consider the graph $G_1$ as figure \ref{Fig1}:

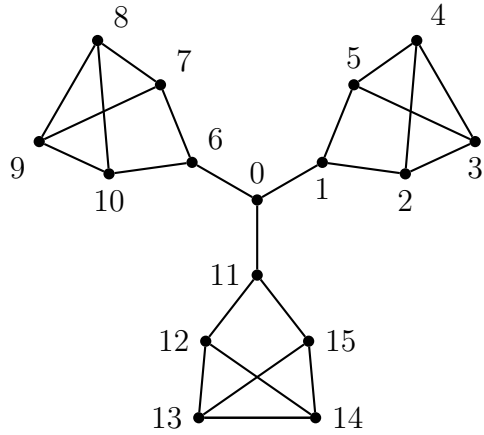
\begin{figure}[h!]
\centering
\begin{tikzpicture}[scale=.5]
\GraphInit[vstyle=Classic]
\tikzset{VertexStyle/.style = {shape = circle,fill = black, minimum size = 1pt,inner sep=1.5pt}}
\Vertex[Lpos=90,L=$0$,a=0,d=0]{0}
\Vertex[Lpos=-90,L=$1$,a=30,d=2]{1}
\Vertex[Lpos=-90,L=$2$,a=10,d=4]{2}
\Vertex[Lpos=-90,L=$3$,a=15,d=6]{3}
\Vertex[Lpos=90,L=$5$,a=50,d=4]{5}
\Vertex[Lpos=60,L=$4$,a=45,d=6]{4}
\Vertex[Lpos=30,L=$6$,a=150,d=2]{6}
\Vertex[Lpos=30,L=$7$,a=130,d=4]{7}
\Vertex[Lpos=30,L=$8$,a=135,d=6]{8}
\Vertex[Lpos=-90,L=$10$,a=170,d=4]{9}
\Vertex[Lpos=-120,L=$9$,a=165,d=6]{10}
\Vertex[Lpos=180,L=$11$,a=270,d=2]{11}
\Vertex[Lpos=180,L=$12$,a=250,d=4]{12}
\Vertex[Lpos=180,L=$13$,a=255,d=6]{13}
\Vertex[Lpos=0,L=$15$,a=290,d=4]{15}
\Vertex[Lpos=0,L=$14$,a=285,d=6]{14}
\Edges(0,1)
\Edges(0,6)
\Edges(0,11)
\Edges(1,2)
\Edges(1,5)
\Edges(2,3)
\Edges(2,4)
\Edges(3,5)
\Edges(3,4)
\Edges(4,5)
\Edges(6,7)
\Edges(6,9)
\Edges(7,8)
\Edges(7,10)
\Edges(8,10)
\Edges(8,9)
\Edges(9,10)
\Edges(11,12)
\Edges(11,15)
\Edges(12,13)
\Edges(15,13)
\Edges(15,14)
\Edges(14,12)
\Edges(14,13)
\Edges(14,13)
\end{tikzpicture}
\caption{The graph $G_1$ with $3$ branches.}
\label{Fig1}
\end{figure}

Clearly $G_1$ has spanning subtree like $T$ with $3$ Leaves and $G$ has no 
spanning subtree with less than $3$ leaves. Every part of graphe $G_1$, like, induce subgraph by the vertices set
$\{1,2,3,4,5 \}$, is called a \textit{branch}. Then, $G_1$ has 3 branches. Let $H$ be a 
branch of the graph $G_1$ with vertices  $\{ 1, 2, 3, 4,5\}$ and edge set
$\{12, 15, 23,24,34,35,45\}$. Since the edge $\{01 \}$ is a cut edge in the graph $G_1$, 
so $T$ must has a vertex with degree one in $V(H)$. Also, in every set-vertex of other branches of the garph $G_1$, the tree 
$T$ must has a vertex with degree one. So, $G_1$ is $3$-ended tree and has no spanning tree with less than 3 leaves.\\
 Now, we counteract $3$-regular graph $G_2$, consider $G_1$ and for each branch of it, like $H$ defined as before, 
 we removed two vertices $\{3 ,4 \}$ and add 8 new vertices $\{v_1, \ldots , v_8 \}$, then we construct a new 3-regular graph with $6$ branches
 as figure $2$, such that the figure \ref{Fig3}, is one of it;s three parts.
 \begin{figure}[h!] 
\centering
\begin{tikzpicture}[scale=.5]
\GraphInit[vstyle=Classic]
\tikzset{VertexStyle/.style = {shape = circle,fill = black, minimum size = 1pt,inner sep=1.5pt}}
\Vertex[Lpos=90,L=$1$,a=0,d=0]{1}
\Vertex[Lpos=-90,L=$2$,a=30,d=2]{2}
\Vertex[Lpos=-90,L=$v_1$,a=10,d=4]{v_1}
\Vertex[Lpos=-90,L=$v_2$,a=15,d=6]{v_2}
\Vertex[Lpos=90,L=$v_4$,a=50,d=4]{v_4}
\Vertex[Lpos=60,L=$v_3$,a=45,d=6]{v_3}
\Vertex[Lpos=30,L=$5$,a=150,d=2]{5}
\Vertex[Lpos=30,L=$v_5$,a=130,d=4]{v_5}
\Vertex[Lpos=30,L=$v_6$,a=135,d=6]{v_6}
\Vertex[Lpos=-90,L=$v_8$,a=170,d=4]{v_7}
\Vertex[Lpos=-120,L=$v_7$,a=165,d=6]{v_8}
\Edges(1,2)
\Edges(1,5)
\Edges(2,v_1)
\Edges(2,v_4)
\Edges(v_1,v_2)
\Edges(v_1,v_3)
\Edges(v_2,v_4)
\Edges(v_2,v_3)
\Edges(v_3,v_4)
\Edges(5,v_5)
\Edges(5,v_7)
\Edges(v_5,v_6)
\Edges(v_5,v_8)
\Edges(v_6,v_8)
\Edges(v_6,v_7)
\Edges(v_7,v_8)
\end{tikzpicture}
\caption{Two branch of the new graph $G_2$, that construct of the graph $G_1$.}
\label{Fig3}
\end{figure}
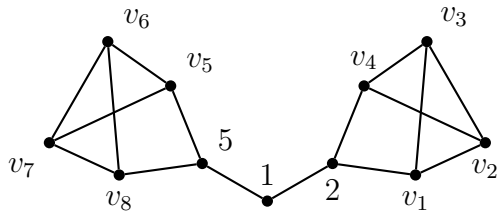
 Clearly $|G_2|=16+3\times6$ and minimum number leaves in every spanning subtree of  $G_2$ is at least $2\times 3$ and obviously $G_2$ has spanning subtree with  $2\times 3$ leaves.\\
Let the number of vertices of $G_m$ is equal $n$ and the number of branches of $G_m$ is equal $k$, then we have the following table:
\begin{table}[h!]
\centering
\begin{tabular}{cccccccccc}
\hline
$m$ & $n$ & $k$ &\\
\hline
$G_1$ &\hspace{1cm} $16$ &\hspace{1cm} $3$ &\\
$G_2$ & \hspace{1cm}$16+3\times6$ &\hspace{1cm} $2\times3$ &\\
$G_3$ &\hspace{1cm} $16+3\times6+2\times3\times6$ &\hspace{1cm} $2\times2\times3$ & \\
$\ldots$ &\hspace{1cm} $\ldots$ & \hspace{1cm}$\ldots$ & \\
$G_m$ &\hspace{1cm} $16+3\times6+\ldots+2^{m-2}\times3\times6$ &\hspace{1cm} $2^{m-1}\times3$ &  \\
\hline
\end{tabular}
\caption{}\label{tab2.3}
\end{table}

It is obvious for each $m\in\mathbb{N}$, If the number of vertices of $G_m$ is equal $n$ and the
number of branches  of $G_m$ is equal $k$, then $\frac{n+2}{6}=k$, and so $G_m$ 
is $\frac{n+2}{6}$-ended tree (such that $\frac{n+2}{6}$ is the minimum number for that $G_m$ has spanning $\frac{n+2}{6}$-ended tree).
\begin{conjecture}
There exist $n\in\mathbb{N} $ such that each $3$-regular graph with at least $n$ vertices has spanning $\lfloor \frac{n+2}{6}\rfloor $-ended tree.
\end{conjecture}

\end{document}